\documentclass[10pt]{amsart}
\usepackage{amsmath,amssymb,verbatim}
\usepackage[latin1]{inputenc}
\usepackage{enumerate}
\usepackage{pdflscape}
\usepackage{amsthm}
\usepackage{mathrsfs}
\usepackage[dvips]{graphicx}
\usepackage{color}
\usepackage[normalem]{ulem}
\usepackage{cancel}
\usepackage{tikz}
\usetikzlibrary{matrix}
\usepackage[all]{xy}
\usepackage{longtable}
\usepackage[norelsize, boxruled]{algorithm2e}

\newtheorem{theorem}{Theorem}[section]
\newtheorem{lemma}[theorem]{Lemma}
\newtheorem{proposition}[theorem]{Proposition}
\newtheorem{example}[theorem]{Example}
\newtheorem{definition}[theorem]{Definition}
\newtheorem{corollary}[theorem]{Corollary}
\newtheorem{remark}[theorem]{Remark}

\newtheorem{remarks}[theorem]{Remarks}
\newtheorem{notation}[theorem]{Notation}

\newcommand{\Cen}{{\rm C}}

\newcommand{\Aut}{\mbox{\rm Aut}}
\newcommand{\Inn}{\mbox{\rm Inn}}

\newcommand{\Z}{{\mathbb Z}}
\newcommand{\Q}{{\mathbb Q}}

\newcommand{\C}{{\mathbb C}}
\newcommand{\F}{{\mathbb F}}

\newcommand{\GL}{{\rm GL}}

\newcommand{\ind}{{\rm ind}}

\newcommand{\matriz}[1]{\begin{array} #1 \end{array}}
\newcommand{\pmatriz}[1]{\left(\begin{array} #1 \end{array}\right)}
\newcommand{\GEN}[1]{\langle #1 \rangle}

\newcommand{\U}{\mathcal{U}}

\newenvironment{proofof}{\par\noindent{\bf Proof of }}{\qed\par\bigskip}

\newcommand{\V}{\mathrm{V}}

\newcommand{\pa}[2]{\varepsilon_{#2}(#1)}

\newcommand{\diag}{\operatorname{diag}}
\newcommand{\Cl}{\operatorname{Cl}}

\newcommand{\qand}{\quad \text{and} \quad}

\SetKwInput{KwData}{Input}
\SetKwInput{KwResult}{Output}

\title[Candidates to counterexamples to the Zassenhaus Conjecture]{An algorithm to construct candidates to counterexamples to the Zassenhaus Conjecture}

\author{Leo Margolis and Ángel del Río}
\thanks{This research is partially supported by the European Commission H2020 program under Grant 705112-ZC, the FWO (Research Foundation Flanders) by the Spanish Government under Grant MTM2016-77445-P with "Fondos FEDER" and, by Fundación Séneca of Murcia under Grant 19880/GERM/15.}



\begin{document}

\begin{abstract}
Let $G$ be a finite group, $N$ a nilpotent normal subgroup of $G$ and let $\mathrm{V}(\mathbb{\Z} G, N)$ denote the group formed by the units of the integral group ring $\mathbb{\Z} G$ of $G$ which map to the identity under the natural homomorphism $\mathbb{\Z} G \rightarrow \mathbb{\Z} (G/N)$.
Sehgal asked whether any torsion element of $\mathrm{V}(\mathbb{\Z} G, N)$ is conjugate in the rational group algebra of $G$ to an element of $G$.
This is a special case of the Zassenhaus Conjecture.

By results of Cliff and Weiss and Hertweck, Sehgal's Problem has a positive solution if $N$ has at most one non-cyclic Sylow subgroup.
We present some algorithms to study Sehgal's Problem when $N$ has at most one non-abelian Sylow subgroup.
They are based on the Cliff-Weiss inequalities introduced by the authors in a previous paper.
With the help of these algorithms we obtain some positive answers to Sehgal's Problem and use them to show that for units in $\mathrm{V}(\mathbb{\Z} G,N)$ our method is strictly stronger than the well known HeLP Method.
We then present a method to use the output of one of the algorithms to construct explicit metabelian groups which are candidates to a negative solution to Sehgal's Problem.
Recently Eisele and Margolis showed that some of the examples proposed in this paper are indeed counterexamples to the Zassenhaus Conjecture. These are the first known counterexamples. Moreover, we prove that every metabelian negative solution of Sehgal's Problem satisfying some minimal conditions is given by our construction.
\end{abstract}

\maketitle

\section{Introduction}

A famous conjecture by H.J.~Zassenhaus about the torsion units of integral group rings states:\\

\textbf{Zassenhaus Conjecture:} If $G$ is a finite group and $u \in \Z G$ is a unit of finite order in the integral group ring of $G$ then $u$ is conjugate in the rational group algebra $\Q G$ of $G$ to an element of the form $\pm g$ for some $g \in G$.\\

Though studied by many authors the conjecture remains open for wide classes of groups. One important result is Weiss' proof of the Zassenhaus Conjecture for nilpotent groups \cite{Weiss1991}. A particularly stubborn class of groups has been the class of metabelian group, see e.g. \cite{MarciniakRitterSehgalWeiss1987, Hertweck2008}, though here it is known the conjecture holds for cyclic-by-abelian groups \cite{CaicedoMargolisdelRio2013} or in some other special cases \cite{MarciniakRitterSehgalWeiss1987, MargolisdelRioCW1}. The main aim of this paper is to construct candidates to counterexamples to the Zassenhaus Conjecture in the class of metabelian groups, cf. Remark~\ref{RemFinal} for the reasons why we believed this to be good candidates.
Motivated by this observations Eisele and Margolis have proved recently that this construction indeed allows to find counterexamples to the conjecture \cite{EiseleMargolis}.

When studying a metabelian group $G$, with an abelian subgroup $N$ containing the commutator subgroup of $G$, proofs of the Zassenhaus Conjecture are usually divided into two parts: The torsion elements in the group $\V(\Z G,N)$ formed by the units mapping to the identity under the linear extension of the homomorphism $G \rightarrow G/N$ to $\Z G \rightarrow \Z (G/N)$ are studied separately from the other units.
This and the result of Weiss probably motivated S.K.~Sehgal to pose the following problem.  \\

\textbf{Sehgal's Problem:} \cite[Problem~35]{Sehgal1993}
Let $N$ be a normal nilpotent subgroup of the finite group $G$ and let $u$ be a torsion element of $\V(\Z G,N)$.
Is $u$ conjugate in the units of $\Q G$ to an element in $G$?\\

In that case we say that \emph{Sehgal's Problem has a positive solution for $G$ and $N$}.
In case $N$ is a nilpotent group we say that \emph{Sehgal' Problem has a positive solution for $N$} if it has a positive solution for every finite group $G$ and every normal subgroup of $G$ isomorphic to $N$.

Using results by Weiss \cite{Weiss1988}, Marciniak and Sehgal \cite{MarciniakSehgal2000} and Hertweck \cite{MargolisHertweck}, it follows that Sehgal's Problem has a positive solution if $N$ is a $p$-group or if $N$ is abelian and $[G:N] \leq 5$.
It is also the focus of \cite{MargolisdelRioCW1} where it is shown how results of Hertweck \cite{MargolisHertweck} and Cliff and Weiss \cite{CliffWeiss} can be combined to give a positive solution if $N$ has at most one non-cyclic Sylow subgroup.

In this paper we study how the methods from \cite{MargolisdelRioCW1} can be used to give algorithmically positive answers to Sehgal's Problem when $N$ has more than one non-cyclic Sylow subgroup. Though the methods presented in \cite{MargolisdelRioCW1} can in principle be applied to any nilpotent $N$ we will concentrate here on the case that $N$ possesses an abelian Hall $p'$-subgroup, since in this case the methods apply in a nicer way and we are especially interested in the class of metabelian groups.
Algorithms~\ref{AlgorithmLocalGN} and \ref{AlgorithmGlobalGN} provide tools to deal with Sehgal's Problem for some concrete $G$ and $N$ and Algorithm~\ref{AlgorithmLocalN} for some $N$ independently of $G$.

We first study in depth the cases where $N$ has two non-cyclic Sylow subgroups isomorphic to elementary-abelian groups of rank 2 and show how much our algorithms can prove here. On the positive side we obtain the following.

\begin{theorem}\label{Positive2-3-5} Assume $N$ is a nilpotent group which has exactly two non-cyclic Sylow subgroups and one of them is an elementary abelian $p$-group of rank $2$ with $p \leq 5$. Then Sehgal's Problem has a positive solution for $N$.
\end{theorem}

This will be a consequence of Propositions~\ref{CyclicFactors} and \ref{23and5}.
We then use this result to show that the algorithms presented here provide strictly stronger restrictions on the the torsion units in $\V(\Z G, N)$ than the restrictions of the well known HeLP Method, cf. Section~\ref{HeLPFails}. This also emphasizes why the orders in the counterexample in the construction of \cite{EiseleMargolis} are not divisible by a prime smaller than 7.

However a result like Theorem~\ref{Positive2-3-5} can not be achieved using our methods for $p \geq 7$, cf. Example~\ref{C7C7Example} and Section~\ref{SecConstruction}. More precisely we provide a general construction of groups for which our methods can not give a positive answer to Sehgal's Problem, cf. Subsection~\ref{SubSecConstrucion}.
This construction is based on the output of Algorithm~\ref{AlgorithmLocalN} and the constructed groups are metabelian such that their derived subgroup is the direct product of two cyclic groups of order $pq$ with $p$ and $q$ different primes.
Using an implementation of our algorithms we produce this output, cf. Tables~\ref{TableCyclic} and \ref{TableNonCyclic}, and realize this construction explicitly.

It turns out that some of the examples obtained in this way gives rise to negative solutions to Sehgal's Problem and hence to the first known counterexample to the Zassenhaus Conjecture.
This has been proved recently by Eisele and Margolis in \cite{EiseleMargolis}.
More precisely, they provide a list of conditions on a finite group $N\rtimes A$, with $N$ abelian, which imply the existence of a torsion unit in $\V(\Z G,N)$ not rationally conjugate to an element of $G$.
The most relevant of these conditions are given by the algorithms presented in this paper.
In fact, the construction of metabelian groups mentioned in the previous paragraph provides groups which satisfy these conditions and indeed their concrete counterexample is one of the groups constructed in Section~\ref{SecConstruction}.

\section{The algorithms}\label{SectionAlgorithms}

To introduce our algorithms and describe their use we need to establish some notation and recall some known results.

\subsection{Notation and preliminaries}

We use standard notation $C_n$, $D_n$ and $Q_n$ for cyclic, dihedral and quaternion groups of order $n$.

Let $g$ be an element of finite order in a group and $N$ be a finite nilpotent group.
For a set of primes $\pi$ write $g_{\pi}$ and $g_{\pi'}$ for the $\pi$-part and $\pi'$-part of $g$ and $N_{\pi}$ and $N_{\pi'}$ for the $\pi$-Hall and $\pi'$-Hall subgroups of $N$.
When $\pi=\{p\}$ we just write $g_p$, $N_p$, $g_{p'}$ or $N_{p'}$.

We use exponential notation for the action of a group on another group and, in particular, for the action of any group on itself by conjugation.

Let $G$ be a group acting on a group $M$. In many of the examples $G$ and $M$ will be subgroups of a given group such that $G$ normalizes $M$ and consider $G$ acting on $M$ by conjugation.
If $S\subseteq M$ then let $S^G = \{s^g : s\in S, g\in G\}$ and $\Cen_G(S)=\{g\in G : s^g=s \text{ for every } s\in S\}$.
In case $S=\{m\}$ we simply write $m^G$ and $\Cen_G(m)$, respectively.
The set $m^G$ is called the $G$-class of $m$ and we have $|m^G|=[G:\Cen_G(m)]$.
Two elements of $M$ are said to be $G$-\emph{conjugate} if they belong to the same $G$-class.
In case the elements have finite order then they are called \emph{locally $G$-conjugate} if their $p$-parts are $G$-conjugate for every prime $p$.
The equivalence class of this relation containing $m$ is called the \emph{local $G$-class} of $m$ and denoted $\ell_G(m)$.
Let $\Cl_G(M)$ denote the set of $G$-classes of $M$.
An integral $G$-class function on $M$ is a map $\varepsilon:M\rightarrow \Z$ which is constant in each $G$-class of $M$. In that case we consider $\varepsilon$ as map defined on $\Cl_G(M)$ by setting $\varepsilon(m^G)=\varepsilon(m)$.
Let  $\Inn_G(M)$ denote the set of automorphisms of $M$ given by $m\mapsto m^g$ for some $g\in G$. Clearly, $S^G=S^{\Inn_G(M)}$ and $\ell_G(m)=\ell_{\Inn_G(M)}(m)$.

Let $G$ be a finite group.
Call a normal subgroup $N$ of $G$ \textit{cocyclic} if $G/N$ is cyclic and call a coset of a cocyclic normal subgroup of $G$ a \textit{cocyclic coset}. A \textit{minimal cocyclic subgroup} of $G$ is a cocyclic normal subgroup not properly containing any other cocyclic normal subgroup of $G$ and a \textit{minimal cocyclic coset} is a coset of a minimal cocyclic subgroup of $G$.

We denote by $\V(\Z G)$ the group of units of augmentation $1$ in $\Z G$.
If an element $u \in \V(\Z G)$ is (locally) conjugate in the rational group algebra $\Q G$ to an element of $G$ one says that $u$ is \emph{(locally) rationally conjugate} to an element of $G$.

Let $u \in \Z G$ and let $u_g$ denote its coefficient at an element $g$ of $G$.
The partial augmentation of $u$ at $g$ (or at $g^G$) is
\[\pa{u}{g} = \pa{u}{g^G} = \sum_{x \in g^G} u_x. \]
This is a fundamental notion in the study of the Zassenhaus Conjecture because if $u$ is a torsion element of $\V(\Z G)$ then $u$ is rationally conjugate to an element in $G$ if and only if $\pa{u^d}{g} \geq 0$ for all $g \in G$ and $d \in \Z$ \cite[Theorem~2.5]{MarciniakRitterSehgalWeiss1987}.
We will use these results freely.

For a finite abelian group $A$ we define
$$h_A = \frac{\sum_{X\in \pi^-} \prod_{p\in \pi} (p^{k_p-\Delta_X(p)}-1)}{\prod_{p\in \pi} (p-1)p^{k_p-1}}$$
where $\pi$ denotes the set of primes dividing $|A|$, $\pi^-$ is the set of subsets of $\pi$ of odd cardinality, $\Delta_X$ denotes the characteristic function of $X$ as a subset of $\pi$ and $k_p$ is the number of factors in the factorization of $A_p$ as a direct product of non-trivial cyclic groups or equivalently $k_p$ is the rank of the socle of $A_p$.
Observe that $h_A=0$ if and only if $A$ is cyclic.
Moreover, if $B$ is cyclic of order coprime with $A$ then $h_A=h_{A\times B}$.

The next proposition collects some useful results about the partial augmentations of torsion units in $\V(\Z G,N)$.
The first one is due to Hertweck (see \cite{MargolisHertweck}), the second one is a consequence of Proposition~2.1 and Theorem~2.3 in \cite{MargolisdelRioPAP} and the third one is \cite[Corollary~3.5]{MargolisdelRioCW1}.
Finally, the fourth one follows from \cite[Lemma~5.1]{MargolisdelRioCW1}.
For that observe that if $m_A$ and $m_A^-$ are defined as in \cite[Section~4]{MargolisdelRioCW1} then $m_A\le m_A^-=|A|h_A$.

\begin{proposition}\label{PAFacts}
Let $G$ be a finite group, let $N$ be a nilpotent normal subgroup of $G$ and let $u$ be a torsion element of $\V(\Z G,N)$.
\begin{enumerate}
\item\label{HertweckPAdic}
The set
    \[\ell_{\Q G}(u) = \{n\in N : u  \text{ is locally rationally conjugate to } n \} \]
contains all $g \in G$ satisfying $\pa{u}{g} \neq 0$.
In particular, $u$ is locally rationally conjugate to an element $n$ of $N$ and hence if either $u$ is a $p$-element or $N$ is a $p$-group then $u$ is rationally conjugate to an element in $N$.

\item\label{Justu}
$u$ is rationally conjugate to an element $n$ of $G$ if and only if  $\pa{u}{g}\ge 0$ for every $g\in G$. In that case $n\in N$.

\item\label{CocyclicInequality} If $N_{p'}$ is abelian then for every cocyclic subgroup $K$ of $N_{p'}$ we have
$$\sum_{g \in nK} |\Cen_G(g)|\pa{u}{g} \geq 0.$$

\item\label{BoundPA} If there is a set of primes $\pi$ such that $A=N_{\pi}$ is abelian and $N_{\pi'}$ has at most one non-cyclic Sylow subgroup then for every $n\in N$ we have
        $$h_A\;[\Cen_G(n_{\pi'}):\Cen_G(n)]+\pa{u}{n}\ge 0.$$
\end{enumerate}
\end{proposition}

Let $G$, $N$ and $u$ be as in Proposition~\ref{PAFacts} and consider $G$ acting on $N$ by conjugation.
Then $n\mapsto \pa{u}{n}$ is an integral $G$-class function on $N$ which vanishes outside of $\ell_{\Q G}(u)$.
Moreover, $\ell_{\Q G}(u)$ is determined by any of its element.
More precisely, if $m\in \ell_{\Q G}(u)$ then $\ell_{\Q G}(u) = \ell_G(m)$.
Fix a prime $p$.
Then the $p$-parts of the elements of $\ell_{\Q G}(u)$ form a $G$-class of $N$.
Fix $m\in \ell_{\Q G}(u)$ and let $H=\Cen_G(m_p)$.
Then $n\mapsto \pa{u}{m_pn}$ defines an integral $H$-class function on $N_{p'}$ which vanishes outside $\ell_G(m_{p'})$.
Next definition formalizes the properties satisfied by these maps when $u$ is not rationally conjugate to an element of $G$ and $N_{p'}=A$ such that $A$ is abelian.

?\begin{definition}\label{DefinitionE}
Let $G$ be a group acting on a finite abelian group $A$ and $H$ be a subgroup of $G$.
Then $\mathcal{E}(G,H,A)$ denotes the set of integral $G$-class functions $f$ of $A$ satisfying the following conditions:
\begin{enumerate}[(I)]
\item\label{SumOne} $\sum_{X\in \Cl_H(A)} f(X) = 1$.

\item\label{VanishOutSideLocalClass}
$f$ vanishes outside a local $G$-class of $A$.

\item\label{Inequation}
$\sum_{c\in C} |\Cen_H(c)| \; f(c)\ge 0$
for each minimal cocyclic coset $C$ of $A$.

\item\label{LowerBound}
If $\pi$ is the set of prime divisors $p$ of $|A|$ for which $A_p$ is cyclic then $f(a)\ge - h_A [\Cen_H(a_{\pi}):\Cen_H(a)]$ for every $a\in A$.

\item\label{Negative} $f(a)<0$ for some $a\in A$.
\end{enumerate}
\end{definition}

\begin{remarks}\label{RemarksOperatorE}
\begin{enumerate}
\item Conditions \eqref{SumOne} and \eqref{LowerBound} imply that $\mathcal{E}(G,H,A)$ is finite.

\item\label{CocyclicNotMinimal} The inequality in \eqref{Inequation} holds for every $f\in \mathcal{E}(G,H,A)$ and every cocyclic coset $C$ of $A$  (i.e not only for the minimal one). Indeed, let $C$ be a $K$-coset of $A$ for $K$ a cocyclic subgroup $A$. Let $L$ be a minimal cocyclic subgroup of $A$ contained in $K$. Then $C$ is the disjoint union of the cosets of $L$ contained in $C$. Thus the inequality in \eqref{Inequation} for $C$ is a consequence of several of such inequalities for $L$-cosets.

\item\label{InequalityAlternative} If $a_X$ denotes a representative of $X$ for each $X\in \Cl_H(A)$, then \eqref{Inequation} is equivalent to $\sum_{X\in \Cl_H(A)}|\Cen_H(a_X)|\cdot |X\cap C| \cdot f(X) \ge 0$.

\item\label{ACyclicEEmpty} If $f\in \mathcal{E}(G,H,a)$ and $f(a)<0$ then $a_{\pi'}\ne 1$ for every set of primes $\pi$ with $A_{\pi}$ cyclic. Otherwise $aA_{\pi'}\cap \ell_G(a)=\{a\}$ and hence applying Condition~\eqref{Inequation} to $C=aA_{\pi'}$ we obtain $|\Cen_G(a)| f(a) = \sum_{c\in C} |\Cen_G(c)| f(c) \ge 0$. In particular, if $A$ is cyclic then $\mathcal{E}(G,H,A)=\emptyset$.
\end{enumerate}
\end{remarks}

By Proposition~\ref{PAFacts}, if Sehgal's Problem has a negative solution for $G$ and $N$ and $N_{p'}$ is abelian then the set $\mathcal{E}(G,\Cen_G(y),N_{p'})$ is not empty for some $y\in N_p$.
So the strategy to attack Sehgal's Problem for $G$ and $N$ with $N$ a nilpotent normal subgroup such that $N_{p'}$ is abelian consists in proving that all these sets are empty.
Actually, these sets provides more information because their elements describe the partial augmentation of each torsion unit providing a negative solution for Sehgal's Problem.
This is exploded by our algorithms to construct minimal groups which are potential counterexamples to Sehgal's Problem and the partial augmentations of the torsion units providing such potential counterexamples.

\subsection{The algorithms}

Using Definition~\ref{DefinitionE} we can formulate algorithms which allow to solve Sehgal's Problem in some situations.

\begin{algorithm}[h]%
\label{AlgorithmLocalGN}
\caption{Sehgal's Problem for $G$ and $N$, with $N_{p'}$ abelian}
 \KwData{$G$, $N$, $p$:\\ \hspace{1cm} $G$ finite group, $N$ normal subgroup of $G$, $p$ prime with $N_{p'}$ abelian.}
 \KwResult{\texttt{true} if the algorithm can solve Seghal's Problem for $G$ and $N$. Otherwise, a list of integral $G$-class functions of $G$ covering all the negative solutions to Sehgal's Problem for $G$ and $N$.}
 Set $L$ to be an empty list. \\
 Let $Y$ be a set of representatives of $G$-classes inside $N_p$\\
 \For{$y\in Y$}{
 	Set $H:=\Cen_G(y)$.\\
 Set $E := \mathcal{E}(G,H, N_{p'})$.\\
For each $f\in E$ add to $L$ the map $\varepsilon:N\rightarrow \Z$ with \\
$\varepsilon(n)=f(n_{p'}^{g})$ if $n_p^g=y$ for some $g\in G$ and $\varepsilon(n)=0$, otherwise.
}
\eIf{$L$ is empty}{\texttt{return true}}{\texttt{return} $L$}			
\end{algorithm}

\begin{theorem}\label{TheoremAbelianHallAlgorithm}
Let $G$ be a group, $N$ a nilpotent normal subgroup of $G$ and $p$ a prime such that $N_{p'}$ is abelian.
If the output of Algorithm~\ref{AlgorithmLocalGN} is \texttt{true} then Sehgal's Problem has a positive solution for $G$ and $N$.
Otherwise, the output of Algorithm~\ref{AlgorithmLocalGN} is formed by a set of integral $G$-class functions of $G$ which contains all the maps $g\mapsto \pa{u}{g}$ with $u$ a torsion element of $\V(\Z G,N)$ which is not rationally conjugate to an element of $G$.

\end{theorem}

\begin{proof}
Let $u$ be a torsion element of $\V(\Z G,N)$ which is not rationally conjugate to an element in $G$ and let $m\in \ell_{\Q G}(u)$.
We have to show that the map $\varepsilon:n\mapsto \pa{u}{n}$ is one of the elements of the output of Algorithm~\ref{AlgorithmLocalGN} with input $G$, $N$ and $p$.
Let $Y$ be the set of representatives of the $G$-classes insided $N$ selected when performing the algorithm.
Then $m_p^G=y^G$ for some $y\in Y$. Let $H=\Cen_G(y)$.
For every $n\in N_{p'}$ let $f(n)=\varepsilon(yn_{p'})$.
Then $f$ is an $H$-class function and $\varepsilon(n)=f(n_{p'}^g)$ if $n_p^g=y$ for some $g\in G$ and $\varepsilon(n)=0$, otherwise.
We will prove that $f\in \mathcal{E}(G,H,N_{p'})$.
This will prove that $\varepsilon$ belongs to the output of the algorithm and the theorem will follow.

So we verify that $f$ satisfies the conditions of Definition~\ref{DefinitionE} for $A=N_{p'}$.
Conditions~\eqref{SumOne} and \eqref{Negative} are clear.
Conditions~\eqref{VanishOutSideLocalClass}, \eqref{Inequation} and \eqref{LowerBound}  are consequence of statements \eqref{HertweckPAdic}, \eqref{CocyclicInequality} and \eqref{BoundPA} of Proposition~\ref{PAFacts}.
\end{proof}

If $N$ is abelian then we can apply Algorithm~\ref{AlgorithmLocalGN} for every prime $p$ dividing the order of $N$ and take the intersection of their outputs.
Thus from Theorem~\ref{TheoremAlgorithmLocalN} we obtain at once the following:

\begin{corollary}
Let $G$ be a group and $N$ an abelian normal subgroup of $G$.
If the output of Algorithm~\ref{AlgorithmGlobalGN} with input $G$ and $N$ is \texttt{true} then Sehgal's Problem has a positive solution for $G$ and $N$.
Otherwise, the output of Algorithm~\ref{AlgorithmGlobalGN} is formed by a set of integral $G$-class functions of $G$ which contains all the maps $g\mapsto \pa{u}{g}$ with $u\in \V(\Z G,N)$ and $u$ not rationally conjugate to an element of $G$.

\begin{algorithm}[h]%
\caption{Sehgal's Problem for $G$ and $N$, with $N$ abelian}\label{AlgorithmGlobalGN}

\KwData{$G$, $N$: \\ \hspace{1cm} $G$ finite group, $N$ normal abelian subgroup of $G$.}

\KwResult{\texttt{true} if the algorithm can solve Seghal's Problem for $G$ and $N$. Otherwise, a list of integral $G$-class functions of $G$ covering all the negative solutions to Sehgal's Problem for $G$ and $N$.}

 Let $p_1,\dots,p_k$ be the primes dividing the order of $N$.\\
 \For{$i=1,\dots,k$}{
 Let $L_{\ell}$ be the output of Algorithm~\ref{AlgorithmLocalGN} with input $G$, $N$ and $p_i$.\\
 Set $L:=\begin{cases}L, & \text{if } i=1; \\ L\cap L_{\ell}, & \text{otherwise}.\end{cases}$\\
 \If{$L$ is empty}{\texttt{break} and\\ \texttt{return true}}
}
 {\texttt{return} $L$}
\end{algorithm}
\end{corollary}

Our third algorithm, with input an abelian group $A$, can give a positive solution to Sehgal's Problem for $N$ (independent of $G$) when $A$ is the Hall $p'$-subgroup of $N$ for some prime $p$.

\begin{theorem}\label{TheoremAlgorithmLocalN}
Let $A$ be a finite abelian group. Then for every finite nilpotent group $N$ such that $A\cong N_{p'}$ for some prime $p$ the following statements hold:
\begin{enumerate}
\item If the output of Algorithm~\ref{AlgorithmLocalN} with input $A$ is \texttt{true} then Sehgal's Problem has a positive solution for $N$.
\item Moreover, if $N$ is a normal subgroup of a finite group $G$ and $\V(\Z G,N)$ contains a torsion element $u$ which is not rationally conjugate to an element of $G$ then there is $n\in \ell_{\Q G}(u)$ and a pair $(\Inn_{\Cen_G(n_p)}(A),E)$ in the output of Algorithm~\ref{AlgorithmLocalN} with input $A$ such that the map defined on $A$ by $a\mapsto \pa{u}{n_pa}$ belongs to $E$. Moreover the partial augmentations of $u$ are completely determined by this map.
\end{enumerate}

\begin{algorithm}[h]\label{AlgorithmLocalN}
 \caption{Sehgal's Problem for $N$, with abelian Hall $p'$-subgroup $A$}
 \KwData{An abelian finite group $A$}
 \KwResult{\texttt{true} if the algorithm can solve Sehgal's Problem for every nilpotent group $N$ containing $A$ as a Hall $p'$-subgroup for some prime $p$.
 	Otherwise, a list formed by pairs $(K,E)$ with $E$ covering all the negative solutions $u$ to Sehgal's Problem for $G$ and $N$ with $N$ as above and $K=\Inn_{\Cen_G(n_{p'})}(A)$ for some $n\in \ell_{\Q G}(u)$.}
 Set $L$ to be an empty list.\\
 Set $S:=\Aut(A)$,\\
 Let $\mathcal{K}$ be a set of representatives of the conjugacy classes of subgroups of $S$.\\
 \For{$K$ in $\mathcal{K}$}
 {Set $E:=\mathcal{E}(S,K,A)$.\\
\If{$E \ne \emptyset$}{
Add $(K,E)$ to $L$.}
}
\eIf{$L$ is empty}{\texttt{return true}}
 {\texttt{return} $L$}			
\end{algorithm}
\end{theorem}

\begin{proofof}\textbf{ Theorem~\ref{TheoremAlgorithmLocalN}}.
The first statement is a consequence of the second one because the output of Algorithm~\ref{AlgorithmLocalN} is either \texttt{true} or a non-empty list.
So suppose that $u$ is a torsion element of $\V(\Z G,N)$ which is not rationally conjugate to an element in $G$.
By Theorem~\ref{TheoremAbelianHallAlgorithm}, the map $x\mapsto \pa{u}{x}$ belongs to the output of Algorithm~\ref{AlgorithmLocalGN} with input $G$, $N$ and $p$.
Then there is $y\in N_p$ and $f\in \mathcal{E}(G,H,A)$ with $H=\Cen_G(y)$ such that $\pa{u}{n}=f(n_{p'}^{g^{-1}})$ if $n_p=y^g$ for some $g\in G$ and $\pa{u}{n}=0$ otherwise.
In particular $y$ is the $p$-part of some element $n$ of $\ell_{\Q G}(u)$.
Let $K=\Inn_H(N_{p'})$, $S=\Aut(A)$ and $E=\mathcal{E}(S,K,A)$.
Then $a^H=a^K$, $|\Cen_H(a)|=|\Cen_K(a)|\;|\Cen_H(A)|$ and $\ell_G(a)\subseteq \ell_S(a)$ for every $a\in A$.
Using this it follows that $\mathcal{E}(G,H,A)\subseteq E$.
Hence $f\in E$.
Thus $n$ and $E$ satisfy the desired property, because $(K,E)$ belongs to the output of Algorithm~\ref{AlgorithmLocalN} with input $A$ and $\pa{u}{n_pa}=f(a)$ for every $a\in A$.
Hence $f$ determines the partial augmentations of $u$ because this describes those at the elements in $\ell_{\Q G}(u)$ and the partial augmentations vanish outside $\ell_{\Q G}(u)$.
\end{proofof}

In fact, by the following result, for positive solutions of Sehgal's Problem we can ignore cyclic Sylow subgroups of $A$ in Algorithm~\ref{AlgorithmLocalN}.

\begin{proposition}\label{CyclicFactors}
Let $A$ be an abelian group and $B$ a cyclic group of order coprime with $A$.
Algorithm~\ref{AlgorithmLocalN} returns \texttt{true} when applied to $A$ if and only if so does the algorithm when applied to $A \times B$.
\end{proposition}

\begin{proof}
Arguing by induction on the primes dividing the order of $B$ we may assume that $B$ is a non-trivial cyclic $p$-group. Let $S_A=\Aut(A)$, $S_B=\Aut(B)$ and $T=\Aut(A\times B)$.
Then $(a,b)^T=a^{S_A}\times b^{S_B}$, $b^{S_B}=\ell_{S_B}(b)$ and $\ell_T(a,b)=\ell_{S_A}(a)\times \ell_{S_B}(b)$ for every $a\in A$ and $b\in B$.
Moreover the minimal cocyclic cosets of $A\times B$ are the sets of the form $C\times \{b\}$ for $b\in B$ and $C$ a minimal cocyclic coset of $A$.
Finally, recall that $h_A=h_{A\times B}$.

Suppose that the output of Algorithm~\ref{AlgorithmLocalN} with input $A$ is not \texttt{true}.
Then $E=\mathcal{E}(S_A,K,A)$ is not empty for some subgroup $K$ of $S_A$.
Let $f\in E$ and let $H$ be the subgroup of $T$ formed by the automorphisms which are the identity on $B$ and whose restriction to $A$ belongs to $K$.
Let $g:A\times B\rightarrow \Z$ be given by $g(a,1)=f(a)$ and $g(a,b)=0$ for $b\ne 1$.
Then it is easy to see that $g\in \mathcal{E}(T,H,A\times B)$.
Thus the output of Algorithm~\ref{AlgorithmLocalN} with input $A\times B$ is not \texttt{true}.

Conversely suppose that the output of Algorithm~\ref{AlgorithmLocalN} with input $A\times B$ is not \texttt{true}.
Thus $\mathcal{E}(T,H,A\times B)$ contains an element $f$ for some subgroup $H$ of $T$.
Let $a\in A$ and $b\in B$ such that $f(ab)\ne 0$, so that $f$ vanishes outside $X=\ell_T(ab)$.
We classify the elements of $X$ by the $H$-classes of their $p$-parts, i.e. we consider a partition $X=\cup_{i=1}^t X_i$ where two elements $x$ and $y$ of $X$ are in the same $X_i$ if and only if $x_p^H=y_p^H$.
For every $i=1,\dots,r$ fix an element $b_i$ which is the $p$-part of an element of $X_i$.
As all the $b_i$'s are $T$-conjugate and $B$ is cyclic, they all generate the same group and hence $K=\Inn_{\Cen_H(b_i)}(A)$ is independent of $i$.
Let $a_{1},\dots,a_t$ be a set of representatives of the $K$-classes inside $X$.
Then $\{a_jb_i : j=1,\dots,t; i=1,\dots,r\}$ is a set of representatives of the $H$-classes inside $X$.
For $i=1,\dots,t$ let $f_i:A\rightarrow \Z$ be defined by $f_i(a)=f(ab_i)$.
Then $f_i$ is an $K$-class function,
	$$1=\sum_{x\in \Cl_H(A\times B)} f(x) =
	\sum_{i=1}^r \sum_{j=1}^t f(a_jb_i) =  \sum_{i=1}^r \sum_{j=1}^t f_i(a_j) = \sum_{i=1}^r \sum_{x\in \Cl_{K}(A)} f_i(x)$$
and, by Remark~\ref{RemarksOperatorE}.\eqref{InequalityAlternative}, the inequality in \eqref{Inequation}, for the definition of $\mathcal{E}(T,H,A\times B)$ takes the form
\begin{equation}\label{DoubleInequation}
\sum_{i=1}^r \sum_{j=1}^t |\Cen_{K}(a_j)| \cdot |(a_jb_i)^H \cap C| \cdot f_i(a_j)\ge 0.
\end{equation}
Recall that this holds for every cocyclic coset $C$ of $A\times B$ (see Remark~\ref{RemarksOperatorE}.\eqref{CocyclicNotMinimal}).

When we say that an $f_i$ satisfies one of the conditions (I)-(V) we refer to the definition of $\mathcal{E}(S_A,K,A)$. Clearly each $f_i$ satisfies condition \eqref{VanishOutSideLocalClass}.

\emph{Claim}: Every $f_i$ satisfies \eqref{Inequation} and \eqref{LowerBound}.
Moreover, there is $i_0\in \{1,\dots,r\}$ such that $f_{i_0}$ satisfies \eqref{SumOne} and for every $i\ne i_0$ the sum of the elements in the image of $f_i$ is $0$.

Indeed, recall that  $h_A=h_{A\times B}$, as $B$ is cyclic of order coprime with $|A|$.
Moreover, if $\pi$ is a set of primes $q$ dividing $|A|$ for which $A_q$ is cyclic and $a\in A$ then $[\Cen_H((ab_i)_{\pi\cup\{p\}}):\Cen_H(ab_i)]= [\Cen_K(a_{\pi}):\Cen_K(a)]$. Hence \eqref{LowerBound} follows.

Let $C$ be a cocyclic coset of $A$.
Then $b_iC$ is a cocyclic coset of $A\times B$,
$|(a_jb_i)^H\cap b_iC|=|a_j^{K}\cap C|$ and
if $i'\ne i$ then $(a_jb_{i'})^H\cap b_iC=\emptyset$.
Therefore \eqref{DoubleInequation} adopts the form
$$\sum_{x\in \Cl_{K}(A)} |x\cap C| \; f_i(x) = \sum_{j=1}^t |\Cen_{K}(a_j)| \cdot |a_j^{K} \cap C| \cdot f_i(a_j)\ge 0.$$
This proves that $f_i$ satisfies condition \eqref{Inequation}.

To prove the last statement of the claim we apply \eqref{DoubleInequation} to $C=b_iA$.
As $(a_jb_{i'})^H\cap C=\emptyset$ if $i\ne i'$ we obtain
\begin{eqnarray*}
	|K|\sum_{x\in \Cl_{K}(A)} f_i(x) &=& |K|\sum_{j=1}^t f_i(a_j) =
	\sum_{j=1}^t |\Cen_{K}(a_j)|\cdot |a_j^{K}| \cdot f_i(a_j) \\
	&=&
	\sum_{j=1}^t |\Cen_H(a_j b_i)|\cdot |(a_jb_i)^H\cap b_iA| \cdot f_i(a_j) \ge 0
\end{eqnarray*}
Thus each sum $\sum_{x\in \Cl_{K}(A)} f_i(x)$ is non negative. As the sum of these sums is 1, we deduce that one of them is $1$ and the others are all $0$.
This finishes the proof of the Claim.

If $f_{i_0}$ satisfies \eqref{Negative} then it belongs to $\mathcal{E}(S_A,K,A)$. Then the output of Algorithm~\ref{AlgorithmLocalN} with input $A$ is not \texttt{true}, as desired.
Otherwise $f_i$ satisfies \eqref{Negative} for some $i\ne i_0$, i.e. $f_i(x)<0$ for some $x\in \Cl_{K}(A)$.
As $\sum_{x\in \Cl_{K}(A)} f_i(x)=0$,  we have $f_i(x_0)>0$ for some $x_0\in \Cl_{K}(A)$.
Let $f':A\rightarrow \Z$ which acts as $f_i$ on $A_i\setminus x_0$ and  $f'(a)=f_i(a)+1$ for every $a\in x_0$. Then $f'\in \mathcal{E}(S_A,K,A)$ and hence the output of Algorithm~\ref{AlgorithmLocalN} is \texttt{true} also in this case.
This finishes the proof.
\end{proof}

For $N$ abelian we could also formulate a global version of Algorithm~\ref{AlgorithmLocalN}, i.e. apply Algorithm~\ref{AlgorithmLocalN} to all Hall $p'$-subgroups of $N$. If one of the outputs is \texttt{true} then Sehgal's Problem has a positive solution for $N$. Otherwise, the algorithm analyzes the compatibility of the different outputs. If there are not compatible outputs then again Sehgal's Problem has a positive solution for $N$.
Otherwise the compatible outputs provides partial augmentations for candidates to counterexamples to Sehgal's Problem for $N$.
We are not going to  formulate this algorithm explicitly.
However we will show some examples in Section~\ref{SecConstruction} of how this algorithm works in some special cases and use it to produce the examples mentioned in the introduction, where the strategy fails.

\section{Applications of the algorithms on small cases}

In this section we analyze the output of Algorithm~\ref{AlgorithmLocalN}. As a consequence we obtain some positive solutions to Sehgal's Problem.
We start collecting a list of properties of elements $f$ in $\mathcal{E}(G,H,A)$ for $G$ a finite group acting on an abelian group $A$ and $H$ a subgroup of $G$, and more concretely of the elements in $a\in A$ with $f(a)<0$.

\begin{lemma}\label{InequalitiesAbelianL}
Let $A$ be a finite abelian group, $G$ a group acting on $A$, $H$ a subgroup of $G$, $q$ a prime integer, $f\in \mathcal{E}(G,H,A)$ and $a\in A$ with $f(a)<0$.
Then the following statements hold:
\begin{enumerate}
\item\label{CocyclicInOrbitL} If $C$ is a cocyclic coset of $A$ then either $C\cap \ell_G(a)=\emptyset$ or $C\cap \ell_G(a)\not\subseteq a^H$.
\item\label{CyclicComplementL} If $A_{q'}$ is cyclic then
\begin{enumerate}

\item\label{AtLeastqL}
$[\Cen_H(a_{q'}) : \Cen_H(a)]\ge q$ and hence $[\Cen_H(a_{q'}) : \Cen_H(a_{q'}) \cap \Cen_H(A_q)]\ge q$.

\item\label{NotMultipleOfqL} If moreover $A_q\cong C_q^2$ then

(i) the center of $\Cen_{\Aut(A)}(A_{q'})$ is not contained in $\Inn_H(A)$ and

(ii) $[\Cen_H(a_{q'}) : \Cen_H(a_{q'}) \cap \Cen_H(A_q)]$ is not divisible by $q$ and, in particular, $[\Cen_H(a_{q'}) : \Cen_H(a)]$ is not divisible by $q$.
\end{enumerate}
\end{enumerate}
\end{lemma}

\begin{proof}
\eqref{CocyclicInOrbitL}
Suppose that $C\cap \ell_G(a)\subseteq a^H$.
If $c\in C$ and $f(c)\ne 0$ then $c\in \ell_G(a)$, by \eqref{VanishOutSideLocalClass}, and hence $c\in a^H$ by assumption. Thus $c^H=a^H$ and, as $f$ is an $H$-class function, we have $f(c)=f(a)$.
Hence, by \eqref{Inequation}, $0\le \sum_{c\in C} |\Cen_G(c)| f(c) = \sum_{c\in C\cap \ell_G(a)} |\Cen_G(c)| f(b) = f(a)\sum_{c\in C\cap \ell_G(a)} |\Cen_G(c)|$. Then $C\cap \ell_G(a)=\emptyset$, since $f(a)<0$.

\eqref{CyclicComplementL}
Suppose that $A_{q'}$ is cyclic.

\eqref{AtLeastqL}
Suppose that $A_q$ is the direct product of $k$ non-trivial cyclic subgroups.
As $\mathcal{E}(G,H,A)\ne\emptyset$, $A_q$ is not cyclic and $a_q\ne 1$, by Remark~\ref{RemarksOperatorE}.\eqref{ACyclicEEmpty}.
Thus $k\ge 2$.
Then
$h_A=\frac{q^{k-1}-1}{q^{k-1}(q-1)}$ and hence,
by \eqref{LowerBound}, we have
$$[\Cen_H(a_{q'}) : \Cen_H(a)]\ge -\frac{f(a)}{h_A}=
-\frac{(q-1)q^{k-1}}{q^{k-1}-1}f(a)\ge \frac{(q-1)q^{k-1}}{q^{k-1}-1}>q-1.$$

\eqref{NotMultipleOfqL} Suppose that $A_q\cong C_q^2$.

(i) Applying \eqref{CocyclicInOrbitL} to $C=a\GEN{a_q}$ we deduce that there is $c\in a\GEN{a_q}\cap \ell_{G}(a)\setminus a^H$, because clearly $a\GEN{a_q}\cap \ell_G(a)\ne\emptyset$.
As $c_{q'}=a_{q'}$ and $c_q$ and $a_q$ are elements of $\GEN{a_q}$ with the same order, we have $c=a_q^ka_{q'}$ for some $k$ coprime with $q$.
Let $\sigma$ be the automorphism of $A$ which acts as the identity on $A_{q'}$ and $\sigma(x)=x^k$ for $x\in A_q$. Then $\sigma$ is central in $\Cen_{\Aut(A)}(A_{q'})$ and $c=a^{\sigma}\not\in a^H$. Thus $\sigma\not\in \Inn_H(A)$.

(ii) Let  $K=\Cen_H(a_{q'})$ and $U = K \cap \Cen_H(A_q)$, a normal subgroup of $K$.
By means of contradiction suppose that $q$ divides $[K:U]$.
Then $K\setminus U$ contains a $q$-element $g$ with $g^q \in U$.
Then $A_q\rtimes \GEN{g}$ is a $q$-group which contains $A_q$ as a normal non-central subgroup.
This implies that $|\Cen_{A_q}(g)|=q$.
By Condition~\eqref{LowerBound} we have
$|a_q^K| =|a^K|= [K:\Cen_H(a)]\ge q=|\Cen_{A_q}(g)|$.
It follows that $a_q^K$ contains an element outside of $\Cen_{A_q}(g)$ and we may assume without loss of generality that $a_q\not \in \Cen_{A_q}(g)$.
Hence $A_q=\GEN{a_q}\times \Cen_{A_q}(g)$ and $a_q^g=a_q^iz$ for some $1\le i<q$ and a generator $z$ of $\Cen_{A_q}(g)$.
Then $a_q=a_q^{g^q} = a_q^{i^q}z^{1+i+\dots+i^{q-1}}$.
So $i\equiv i^q\equiv 1 \mod q$ and therefore $i=1$.
This implies that $a_{p'}\GEN{z}\subseteq a_{p'}^{\Cen_H(a_p)}$.
Then $\pa{u}{a}\ge 0$, by \eqref{CocyclicInOrbitL}. This contradicts the choice of $a$.
\end{proof}

Combining Theorem~\ref{TheoremAbelianHallAlgorithm} with Lemma~\ref{InequalitiesAbelianL} we obtain the following.

\begin{corollary}
If $\pa{n}{u}<0$ for some torsion element $u$ of $\V(\Z G,N)$ with $N$ a nilpotent subgroup of $G$ such that $N_{p'}$ is abelian and $n\in N$ then statements (1) and (2) of Lemma~\ref{InequalitiesAbelianL} holds for $A=N_{p'}$, $G$ acting on $A$ by conjugation, $a=n_{p'}$ and $H=\Cen_G(n_p)$.
\end{corollary}

\begin{proposition}\label{23and5}
Let $A$ be an abelian group which has a cocyclic Sylow subgroup isomorphic to $C_q\times C_q$ with $q=2,3$ or $5$.
Then the output of Algorithm~\ref{AlgorithmLocalN} with input $A$ is \texttt{true}.
\end{proposition}

\begin{proof}
Let $S=\Aut(A)$.
By Proposition~\ref{CyclicFactors} it is enough to prove the statement for $A=C_q\times C_q$ and for that we have to show that $E(K)=\mathcal{E}(S,K,A)$ is empty for $K$ running on a set of representatives of conjugacy classes of subgroups of $S$.
To transfer the calculation to linear algebra we identify $A$ with the underlying additive group of $\F_q^2$ and hence we identify $S$ and $\GL(2,q)$.
Lemma~\ref{InequalitiesAbelianL} implies that this holds if $K$ contains the center of $S$ or $|K|$ is smaller than $q$ or $|K|$ is divisible by $q$.
Of course if the action of $K$ on $A$ is transitive then $E(K)=\emptyset$.
This allow us to display the elements of $A$ in an $\F_q$-plane on which we visualize the minimal cocyclic cosets as lines.

If $q=2$ or $q=3$ then any subgroup of $S$ not containing its center and of order greater than $q$ has order divisible by $q$.
So we only have to consider the case $q=5$.
Observe that there are, up to conjugacy, 7 non-transitive subgroups in $S$ not containing its center and of order greater than $5$ and not divisible by $5$, namely those listed in Table~\ref{Ks}.

\begin{table}[h]
\begin{align*}
K_1 &= \left\langle\begin{pmatrix} 0 & 1 \\ -1 & 1 \end{pmatrix} \right\rangle,
&
K_2 &= \left\langle\begin{pmatrix} -1 & 1 \\ 0 & 1 \end{pmatrix}, \begin{pmatrix} -1 & 1 \\ -1 & 0 \end{pmatrix} \right\rangle,
\\
K_3 &= \left\langle\begin{pmatrix} -2 & -2 \\ 0 & 2 \end{pmatrix}, \begin{pmatrix} -2 & 0 \\ -1 & 2 \end{pmatrix} \right\rangle,
&
K_4 &= \left\langle\begin{pmatrix} 1 & 0 \\ -2 & -1 \end{pmatrix}, \begin{pmatrix} 1 & 1 \\ 0 & -1 \end{pmatrix} \right\rangle,
\\
K_5 &= \left\langle\begin{pmatrix} -1 & 0 \\ 0 & -1 \end{pmatrix}, \begin{pmatrix} 2 & 0 \\ -1 & 1 \end{pmatrix} \right\rangle,
&
K_6 &= \left\langle\begin{pmatrix} 2 & -2 \\ 0 & -2 \end{pmatrix}, \begin{pmatrix} 0 &  -1 \\ 1 & -1 \end{pmatrix} \right\rangle,
\\
K_7 &= \left\langle\begin{pmatrix} 1 & -1 \\ 0 & -1 \end{pmatrix}, \begin{pmatrix} 0 & 1 \\ -1 & 1 \end{pmatrix} \right\rangle.
\end{align*}
\caption{\label{Ks}}
\end{table}

\begin{figure}[h]
	$$\matriz{{cccc}
		\boxed{\xygraph{
				!{<0cm,0cm>;<0.4cm,0cm>:<0cm,0.4cm>::}
				!{(-2,-2)}*+{4}
				!{(-1,-2)}*+{4}
				!{(0,-2)}*+{2}
				!{(1,-2)}*+{3}
				!{(2,-2)}*+{2}
				!{(-2,-1)}*+{4}
				!{(-1,-1)}*+{3}
				!{(0,-1)}*+{1}
				!{(1,-1)}*+{1}
				!{(2,-1)}*+{3}
				!{(-2,0)}*+{2}
				!{(-1,0)}*+{1}
				!{(1,0)}*+{1}
				!{(2,0)}*+{2}
				!{(-2,1)}*+{3}
				!{(-1,1)}*+{1}
				!{(0,1)}*+{1}
				!{(1,1)}*+{3}
				!{(2,1)}*+{4}
				!{(-2,2)}*+{2}
				!{(-1,2)}*+{3}
				!{(0,2)}*+{2}
				!{(1,2)}*+{4}
				!{(2,2)}*+{4}
		}} &
		\boxed{\xygraph{
				!{<0cm,0cm>;<0.4cm,0cm>:<0cm,0.4cm>::}
				!{(-2,-2)}*+{6}
				!{(-1,-2)}*+{6}
				!{(0,-2)}*+{\mathbf{2}}
				!{(1,-2)}*+{5}
				!{(2,-2)}*+{\mathbf{4}}
				!{(-2,-1)}*+{6}
				!{(-1,-1)}*+{5}
				!{(0,-1)}*+{\mathbf{1}}
				!{(1,-1)}*+{\mathbf{3}}
				!{(2,-1)}*+{5}
				!{(-2,0)}*+{\mathbf{4}}
				!{(-1,0)}*+{\mathbf{3}}
				!{(1,0)}*+{\mathbf{1}}
				!{(2,0)}*+{\mathbf{2}}
				!{(-2,1)}*+{5}
				!{(-1,1)}*+{\mathbf{1}}
				!{(0,1)}*+{\mathbf{3}}
				!{(1,1)}*+{5}
				!{(2,1)}*+{6}
				!{(-2,2)}*+{\mathbf{2}}
				!{(-1,2)}*+{5}
				!{(0,2)}*+{\mathbf{4}}
				!{(1,2)}*+{6}
				!{(2,2)}*+{6}
		}}
		&
		\boxed{\xygraph{
				!{<0cm,0cm>;<0.4cm,0cm>:<0cm,0.4cm>::}
				!{(-2,-2)}*+{\mathit{1}}
				!{(-1,-2)}*+{\mathit{3}}
				!{(0,-2)}*+{\mathit{2}}
				!{(1,-2)}*+{\mathit{2}}
				!{(2,-2)}*+{\mathit{3}}
				!{(-2,-1)}*+{\mathit{2}}
				!{(-1,-1)}*+{\mathit{1}}
				!{(0,-1)}*+{\mathit{2}}
				!{(1,-1)}*+{\mathit{3}}
				!{(2,-1)}*+{\mathit{3}}
				!{(-2,0)}*+{\mathit{1}}
				!{(-1,0)}*+{\mathit{1}}
				!{(1,0)}*+{\mathit{1}}
				!{(2,0)}*+{\mathit{1}}
				!{(-2,1)}*+{\mathit{3}}
				!{(-1,1)}*+{\mathit{3}}
				!{(0,1)}*+{\mathit{2}}
				!{(1,1)}*+{\mathit{1}}
				!{(2,1)}*+{\mathit{2}}
				!{(-2,2)}*+{\mathit{3}}
				!{(-1,2)}*+{\mathit{2}}
				!{(0,2)}*+{\mathit{2}}
				!{(1,2)}*+{\mathit{3}}
				!{(2,2)}*+{\mathit{1}}
		}}	
		&
		\boxed{\xygraph{
				!{<0cm,0cm>;<0.4cm,0cm>:<0cm,0.4cm>::}
				!{(-2,-2)}*+{\mathbf{2}}
				!{(-1,-2)}*+{\mathit{5}}
				!{(0,-2)}*+{\mathbf{4}}
				!{(1,-2)}*+{\mathbf{4}}
				!{(2,-2)}*+{\mathit{5}}
				!{(-2,-1)}*+{\mathbf{3}}
				!{(-1,-1)}*+{\mathbf{1}}
				!{(0,-1)}*+{\mathbf{3}}
				!{(1,-1)}*+{\mathit{5}}
				!{(2,-1)}*+{\mathit{5}}
				!{(-2,0)}*+{\mathbf{2}}
				!{(-1,0)}*+{\mathbf{1}}
				!{(1,0)}*+{\mathbf{1}}
				!{(2,0)}*+{\mathbf{2}}
				!{(-2,1)}*+{\mathit{5}}
				!{(-1,1)}*+{\mathit{5}}
				!{(0,1)}*+{\mathbf{3}}
				!{(1,1)}*+{\mathbf{1}}
				!{(2,1)}*+{\mathbf{3}}
				!{(-2,2)}*+{\mathit{5}}
				!{(-1,2)}*+{\mathbf{4}}
				!{(0,2)}*+{\mathbf{4}}
				!{(1,2)}*+{\mathit{5}}
				!{(2,2)}*+{\mathbf{2}}
		}}
		\\
		K_1\cong C_6 & K_2\cong D_6 & K_3\cong Q_8 & K_4 \cong D_8\\
		\boxed{\xygraph{
				!{<0cm,0cm>;<0.4cm,0cm>:<0cm,0.4cm>::}
				!{(-2,-2)}*+{\mathbf{3}}
				!{(-1,-2)}*+{5}
				!{(0,-2)}*+{5}
				!{(1,-2)}*+{5}
				!{(2,-2)}*+{5}
				!{(-2,-1)}*+{4}
				!{(-1,-1)}*+{\mathbf{2}}
				!{(0,-1)}*+{4}
				!{(1,-1)}*+{4}
				!{(2,-1)}*+{4}
				!{(-2,0)}*+{\mathbf{1}}
				!{(-1,0)}*+{\mathbf{1}}
				!{(1,0)}*+{\mathbf{1}}
				!{(2,0)}*+{\mathbf{1}}
				!{(-2,1)}*+{4}
				!{(-1,1)}*+{4}
				!{(0,1)}*+{4}
				!{(1,1)}*+{\mathbf{2}}
				!{(2,1)}*+{4}
				!{(-2,2)}*+{5}
				!{(-1,2)}*+{5}
				!{(0,2)}*+{5}
				!{(1,2)}*+{5}
				!{(2,2)}*+{\mathbf{3}}
		}}
		&
		\boxed{\xygraph{
				!{<0cm,0cm>;<0.4cm,0cm>:<0cm,0.4cm>::}
				!{(-2,-2)}*+{\mathit{2}}
				!{(-1,-2)}*+{\mathit{2}}
				!{(0,-2)}*+{\mathit{1}}
				!{(1,-2)}*+{\mathit{2}}
				!{(2,-2)}*+{\mathit{1}}
				!{(-2,-1)}*+{\mathit{2}}
				!{(-1,-1)}*+{\mathit{2}}
				!{(0,-1)}*+{\mathit{1}}
				!{(1,-1)}*+{\mathit{1}}
				!{(2,-1)}*+{\mathit{2}}
				!{(-2,0)}*+{\mathit{1}}
				!{(-1,0)}*+{\mathit{1}}
				!{(1,0)}*+{\mathit{1}}
				!{(2,0)}*+{\mathit{1}}
				!{(-2,1)}*+{\mathit{2}}
				!{(-1,1)}*+{\mathit{1}}
				!{(0,1)}*+{\mathit{1}}
				!{(1,1)}*+{\mathit{2}}
				!{(2,1)}*+{\mathit{2}}
				!{(-2,2)}*+{\mathit{1}}
				!{(-1,2)}*+{\mathit{2}}
				!{(0,2)}*+{\mathit{1}}
				!{(1,2)}*+{\mathit{2}}
				!{(2,2)}*+{\mathit{2}}
		}}
		&
		\boxed{\xygraph{
				!{<0cm,0cm>;<0.4cm,0cm>:<0cm,0.4cm>::}
				!{(-2,-2)}*+{4}
				!{(-1,-2)}*+{4}
				!{(0,-2)}*+{2}
				!{(1,-2)}*+{3}
				!{(2,-2)}*+{2}
				!{(-2,-1)}*+{4}
				!{(-1,-1)}*+{3}
				!{(0,-1)}*+{1}
				!{(1,-1)}*+{1}
				!{(2,-1)}*+{3}
				!{(-2,0)}*+{2}
				!{(-1,0)}*+{1}
				!{(1,0)}*+{1}
				!{(2,0)}*+{2}
				!{(-2,1)}*+{3}
				!{(-1,1)}*+{1}
				!{(0,1)}*+{1}
				!{(1,1)}*+{3}
				!{(2,1)}*+{4}
				!{(-2,2)}*+{2}
				!{(-1,2)}*+{3}
				!{(0,2)}*+{2}
				!{(1,2)}*+{4}
				!{(2,2)}*+{4}
		}} \\
		K_5\cong C_4\times C_2 & K_6\cong C_3\rtimes C_4 & K_7\cong D_{12}
	}$$
	\caption{\label{Orbits5} $K_i$-classes of the groups $K_i$. The rows and columns are displayed in the following order: $-2,-1,0,1,2$.}
\end{figure}

Figure~\ref{Orbits5} displays the $K_i$-classes of the action of each of these groups on $\F_5^2$.
We consider the images of an element $\varepsilon$ in $E(K_i)$ as variables $x_1,\dots,x_r$, where $r$ is the number of $K_i$-classes, and represent $E(K_i)$ in terms of properties of these variables given by conditions in Definition~\ref{DefinitionE} and Lemma~\ref{InequalitiesAbelianL}.
Actually, we will only calculate properties implied by Conditions \eqref{SumOne} and \eqref{Inequation} and Lemma~\ref{InequalitiesAbelianL}.
For example, Condition~\eqref{SumOne} corresponds to $\sum_{i=1}^r x_j=1$ and Lemma~\ref{InequalitiesAbelianL}.\eqref{AtLeastqL} means that if the $j$-th $K_i$-class contains less than $5$ entries then $x_j\ge 0$. These $K_i$-classes are displayed in boldface in Figure~\ref{Orbits5}.
In case the $j$-th $K_i$-class contains a line then Lemma~\ref{InequalitiesAbelianL}.\eqref{CocyclicInOrbitL} implies $x_j\ge 0$.
We display these $K_i$-classes in italics.
To prove that $E(K_i)$ is empty we prove that the set of solutions does not contain any negative entry.
For example, this is easy for $K_3, K_4$ and $K_6$ since each $K_i$-class of these groups either contains a line (italics) or has cardinality smaller then $5$ (boldface) and then, in these cases, each indeterminate $x_j$ is non-negative.
The conditions defining the remaining sets $E(K_i)$ are below.

\[
E(K_1) \left\{
\matriz{{rrrrrrrrrl}
	x_1 &+& x_2 &+& x_3 &+& x_4 &=& 1 \\
	x_1 &+& x_2 & & & & &\geq& 0  \\
	2x_1 & & &+& 2x_3 &+& x_4 &\geq& 0  \\
	& & 2x_2 &+& x_3 &+& 2x_4 &\geq& 0 \\
	& & & & x_3 &+& x_4 &\geq& 0   \\
	2x_1 &+& x_2 & & &+& 2x_4 &\geq& 0 \\
	x_1 &+& 2x_2 &+& 2x_3 & & &\geq& 0
}\right.\]

\[E(K_2) \left\{
\matriz{{rrrrrrrrrrrrrl}
	x_1 &+& x_2 &+& x_3 &+& x_4 &+& x_5 &+& x_6 &=& 1 \\
	x_1& & & & & & & &  & & &\geq& 0 \\
	& & x_2& & & & & & & & &\geq& 0 \\
	& & & & x_3 & & & & & & &\geq& 0 \\
	& & & & & & x_4 & & & & &\geq& 0 \\
	&&  & & & & & & x_5& + & x_6&\geq& 0 \\
	x_1 & &  & + & x_3 & & &+& 2x_5& + & x_6&\geq& 0 \\	
	& & x_2 & &  &+& x_4 &+& x_5 &+& 2x_6 &\ge & 0 \\
	2x_1 &+& x_2 &&  &&  &&  &+& 2x_6 &\ge & 0\\
	&&  && 2x_3 &+& x_4 &&  &+& 2x_6 &\ge & 0\\
	&+& 2x_2 &+& x_3 &&  &+& 2x_5 &&  &\ge & 0 \\
	x_1 &&  &&  &+& 2x_4 &+& 2x_5 &&  &\ge & 0
}\right.\]

\[E(K_5) \left\{
\matriz{{rrrrrrrrrrrl}
	x_1 &+& x_2 &+& x_3 &+& x_4 &+& x_5 &=1 \\
	x_1& &  & & & & & & &\geq& 0 \\
	& & x_2 & & & & & & &\geq& 0 \\
	& & & & x_3& & & & &\geq& 0 \\
	&& & &  & & x_4&+&x_5&\geq& 0 \\
	&& x_2 & & &+ &4x_4 & & &\geq& 0\\
	& & & & x_3 &&  & +& 4x_5&\geq& 0\\
}\right.\]

\[E(K_7) \left\{
\matriz{{rrrrrrrrrl}
	x_1 &+& x_2 &+& x_3 &+& x_4 &=& 1 \\
	x_1 &+& x_2 & & & & &\geq& 0 \\
	& & & & x_3 &+& x_4 &\geq& 0 \\
	2x_1 & & &+& 2x_3 &+& x_4 &\geq& 0\\
	& & 2x_2 &+& x_3 &+& 2x_4 &\geq& 0\\
	2x_1 &+& x_2 & & &+& 2x_4 &\geq& 0\\
	x_1 &+& 2x_2 &+& 2x_3 & & &\geq& 0
}\right.\]

The reader can easily verify that in all the cases the solutions are trivial, i.e. one of the variables is $1$ and the others are $0$.
\end{proof}

However, the bound $q\leq 5$ in Theorem~\ref{Positive2-3-5} is sharp regarding our methods as the following example shows.

\begin{example}\label{C7C7Example}
The output of Algorithm~\ref{AlgorithmLocalN} with input $C_7\times C_7$ is not \texttt{true}.
\end{example}

\begin{proof}
Let $A=C_7\times C_7$ and $S=\Aut(A)$. To describe a subgroup $K$ of $S$ for which $\mathcal{E}(S,K,A)\ne \emptyset$  we identify $A$ with the additive group of $\F_{49}$.
Then $K$ is the cyclic group generated by the automorphism $x\mapsto \beta x$ for $\beta$ an element of $\U(\F_{49})$ of order $16$.
The base of $\F_{49}$ over $\F_7$ used to display this vector space is $1$ and $\alpha$, where $\alpha$ is a root of $X^2 - X + 3$, a primitive element of $\F_{49}$. Thus $\beta=\alpha^3 = -2\alpha - 3$.
This allows us to mark the orbits of the action of $\alpha^3$ on $\mathbb{F}_7^2$ in Figure~\ref{C7C7Action} together with the inequalities of Condition \eqref{Inequation}.
Conditions \eqref{SumOne} and \eqref{LowerBound} take the form $x_1+x_2+x_3=1$ and  $x_1,x_2,x_3\ge -\frac{16}{7}$ respectively.
One readily verifies that
\[\mathcal{E}(S,K,A) = \{(2,-1,0),(0,2,-1),(-1,0,2)\} \]
and hence the output of Algorithm~\ref{AlgorithmLocalN} with input $C_7\times C_7$ is not \texttt{true}.
\begin{figure}[h!]
$$
\matriz{{cc}
\boxed{\xygraph{
!{<0cm,0cm>;<-2.5cm,0cm>:<0cm,-2.5cm>::}
!{(0,0)}*+{2}
!{(0.16,0)}*+{3}
!{(0.32,0)}*+{2}
!{(0.48,0)}*+{1}
!{(0.64,0)}*+{2}
!{(0.80,0)}*+{2}
!{(0.96,0)}*+{1}
!{(0,.16)}*+{1}
!{(0.16,.16)}*+{1}
!{(0.32,.16)}*+{1}
!{(0.48,.16)}*+{3}
!{(0.64,.16)}*+{2}
!{(0.80,.16)}*+{3}
!{(0.96,.16)}*+{1}
!{(0,.32)}*+{1}
!{(0.16,.32)}*+{3}
!{(0.32,.32)}*+{3}
!{(0.48,.32)}*+{2}
!{(0.64,.32)}*+{2}
!{(0.80,.32)}*+{3}
!{(0.96,.32)}*+{3}
!{(0,.48)}*+{3}
!{(0.16,.48)}*+{2}
!{(0.32,.48)}*+{1}
!{(0.64,.48)}*+{1}
!{(0.80,.48)}*+{2}
!{(0.96,.48)}*+{3}
!{(0.96,.96)}*+{2}
!{(0.80,.96)}*+{3}
!{(0.64,.96)}*+{2}
!{(0.48,.96)}*+{1}
!{(0.32,.96)}*+{2}
!{(0.16,.96)}*+{2}
!{(0,.96)}*+{1}
!{(0.96,.80)}*+{1}
!{(0.80,.80)}*+{1}
!{(0.64,.80)}*+{1}
!{(0.48,.80)}*+{3}
!{(0.32,.80)}*+{2}
!{(0.16,.80)}*+{3}
!{(0,.80)}*+{1}
!{(0.96,.64)}*+{1}
!{(0.80,.64)}*+{3}
!{(0.64,.64)}*+{3}
!{(0.48,.64)}*+{2}
!{(0.32,.64)}*+{2}
!{(0.16,.64)}*+{3}
!{(0,.64)}*+{3}
}}
&
\matriz{{rrrrrrrrrr}
\\\\
\rightarrow &2x_1 &+& 4x_2 &+& x_3 & \ge & 0 \\
\rightarrow &4x_1 &+& x_2 &+& 2x_3 & \ge & 0 \\
\rightarrow &x_1 &+& 2x_2 &+& 4x_3 & \ge & 0
}
}
$$
\caption{Inequalities provided by cyclic action of order 16 on $C_7 \times C_7$.}\label{C7C7Action}	
\end{figure}
\end{proof}

\section{One example where HeLP fails and our method works}\label{HeLPFails}

The goal of this section is to show that our algorithms provide stronger results than the standard HeLP Method giving an explicit example.
The HeLP Method provides restrictions on the partial augmentations of torsion units which can be used in some cases to prove that a torsion element $u$ of $\V(\Z G)$ is rationally conjugate to an element of $G$, by showing that the restrictions imply that the partial augmentations of all the powers of $u$ are non-negative. This method was introduced in \cite{LutharPassi1989} and \cite{HertweckBrauer}.

\begin{remark}\label{HeLPRemark}
In case $N$ is a nilpotent normal subgroup in $G$ and $u$ is a torsion element in $\V(\Z G,N)$ we only need to prove that the partial augmentations of $u$ are non-negative and the calculations of the HeLP Method can be simplified as explained in \cite{MargolisdelRioPAP}. Moreover by \cite[Remark~3.4]{MargolisdelRioPAP} the restrictions provided in this case by the HeLP Method are exactly those provided by the inequalities in \cite[Proposition~3.3]{MargolisdelRioCW1}. By \cite[Remark 3.6]{MargolisdelRioCW1} these restrictions follow from the restrictions in our algorithms. More precisely they follow from the restrictions in \cite[Theorem~3.4]{MargolisdelRioCW1}, which are just Conditions \eqref{SumOne}, \eqref{VanishOutSideLocalClass} and \eqref{Inequation}. So the restrictions provided by our algorithms are at least as strong as those provided by the HeLP Method and it will follow from Example~\ref{StrongerThanHeLP} that they are actually strictly stronger.
\end{remark}

We start recalling some notation and introducing a family of groups which will be encountered later.

If $\chi$ is the character of $G$ afforded by a complex representation $\rho$, $u$ is an element of $\C G$ and $\xi$ is a complex number then let $\mu(\chi,u,\xi)$ denote the multiplicity of $\xi$ as eigenvalue of $\rho(u)$.

\begin{notation}\label{Gpqd}
Fix two different prime integers $p$ and $q$ and let $\alpha$ and $\beta$ be a primitive elements of $\F_{p^2}$ and $\F_{q^2}$ respectively.
Fix a common divisor $d$ of $p^2-1$ and $q^2-1$.
Let $N$ be the underlying additive group of $\F_{p^2}\times \F_{q^2}$ and let $H=\GEN{a,b,c}$ where $a,b$ and $c$ are the automorphisms of $N$ given by
	$$(x,y)^a = (\alpha^d x,y), \quad (x,y)^b = (x,\beta^d y), \quad (x,y)^c = (\alpha x, \beta y),$$
for $x\in \F_{p^2}$ and $y\in \F_{q^2}$.
Then $H$ is the abelian group given by the following presentation (as abelian group)
	 $$H=\GEN{a,b,c| a^{\frac{p^2-1}{d}}=b^{\frac{q^2-1}{d}}=1, c^d=ab}$$
Let $G_d(\alpha,\beta)$ denote the semidirect product $N\rtimes H$.
We call these kind of groups groups of type $G_d(p,q)$.
\end{notation}

Let $G=G_d(\alpha,\beta)$, a group of type $G_d(p,q)$, and let $u$ be a torsion elements of $\V(\Z G,N)$.
If the order of $u$ is different from $pq$ then $u$ is rationally conjugate to an element of $N$ by Proposition~\ref{PAFacts}.\eqref{HertweckPAdic}.
Suppose that $|u|=pq$. Then $\ell_{\Q G}(u)$ is formed by all the elements of $N$ of order $pq$ and the elements of the form $m_i=(1,\beta^i)$ with $i=0,1,\dots,d-1$ form a set of representatives of the $G$-classes contained in $\ell_{\Q G}(u)$.
Hence $\pa{u}{m_0},\dots,\pa{u}{m_{d-1}}$ are the only possible non-zero partial augmentation of $u$.
The HeLP Method provides restrictions on these numbers.

Fix a complex primitive $p$-th root of unity $\zeta_p$ and a complex primitive $q$-th root of unity $\zeta_q$.
Consider the linear character $\psi_0$ of $N$ with kernel the group generated by $(1,1)$ and mapping $(\alpha,\beta)$ to $\zeta_p\zeta_q$.
Set $\chi_0=\ind_N^G(\psi_0)$, the character of $G$ induced by $\psi_0$.

\begin{lemma}\label{HeLPGpqd}
Let $G$ be a group of type $G_d(p,q)$ such that $d$ is coprime with $\gcd(p+1,q+1)$. A list of integers $(x_0,x_1,\dots,x_{d-1})$ satisfies the restrictions on $\pa{u}{m_0},\dots,\pa{u}{m_{d-1}}$ provided by the HeLP Method for an element $u$ of order $pq$ in $\V(\Z G,N)$ if and only if
    \begin{equation}\label{HeLPInequality}
    x_0+\dots+x_{d-1}=1 \qand \sum_{i=0}^{d-1} \mu(\chi_0,m_i,\xi)\; x_i\ge 0
    \end{equation}
for every $pq$-th root of unity $\xi$.
\end{lemma}

\begin{proof}
Suppose that $G=G_d(\alpha,\beta)$.
We first prove that the condition on $d$ implies that $H$ acts transitively on the cyclic subgroups of order $pq$ in $N$. Indeed let $(\alpha^i, \beta^j)$ be an element in $N$.
By the assumption $\gcd(\gcd(d,p+1),\gcd(d,q+1))=1$ and hence, by the Chinese Remainder Theorem there is an integer $r$ such that $r\equiv -i \bmod \gcd(d,p+1)$ and $r\equiv -j \bmod \gcd(d,q+1)$.
Then the two congruence equations $dX \equiv -i-r\bmod (p+1)$ and $dX\equiv -j-r \mod (q+1)$ have solutions $s$ and $t$, respectively.
Thus $p+1 \mid i+r+ds$ and $q+1 \mid j+r+dt$ and therefore
$$(\alpha^i,\beta^j)^{c^ra^sb^t} = (\alpha^{i+r+ds}, \beta^{j+r+dt}) \in \mathbb{F}_p\times \mathbb{F}_q
 .$$
This proves that $c^ra^sb^t$ maps the cyclic group generated by $(\alpha^i, \beta^j)$ to the cyclic group $\F_p\times \F_q$ in $N$. Hence the claim follows.

By \cite[Proposition~3.3]{MargolisdelRioPAP} the restrictions for the HeLP Method are $x_0+\dots+x_{d-1}=1$ and for every character $\chi$ and every $pq$-th root of unity $\xi$
	$$\sum_{i=0}^{d-1} \mu(\chi,m_i,\xi) \; x_i\ge 0.$$
First of all we only need to consider irreducible characters of $G$ and all of them are of the form $\chi=\ind_U^G(\psi)$ where $U$ is a subgroup of $G$ containing $N$ and $\psi$ is a linear character of $U$ \cite[Corollary 3.5.13]{GRG1}.
If the kernel of $\psi$ contains $N_p$ or $N_q$ then $\chi$ takes the same values on all $G$-classes in $\ell_{\Q G}(u)$, because the $p$-parts of the elements of $\ell_{\Q G}(u)$ are conjugate in $G$, and the same holds for its $q$-parts.
Thus the intersection of the kernel of $\psi$ and $N$ is a cyclic group of order $pq$. Since the centralizer of any cyclic group of order $pq$ in $G$ is exactly $N$ and $\psi$ is a linear character this implies that $U = N$.
Since the action of $G$ on the cyclic groups of order $pq$ in $N$ is transitive all the characters we obtain by this construction are Galois-conjugate by \cite[Problem 3.4.4]{GRG1}.
Since Galois-conjugate characters provide the same restrictions for the HeLP Method it thus is sufficient to the study the restrictions provided by the HeLP Method only for one such $\chi$, for example $\chi_0$.
\end{proof}

By Theorem~\ref{Positive2-3-5}, Sehgal's Problem has a positive solution for group of type $G_d(p,q)$ if $p$ or $q$ is at most $5$. However, this cannot be shown using only the HeLP Method as the following example shows:

\begin{example}\label{StrongerThanHeLP}
The HeLP Method fails to give a positive answer to Sehgal's Problem for groups of type $G_3(5,7)$.

Hence, by Theorem~\ref{Positive2-3-5} and Remark~\ref{HeLPRemark}, the restrictions provided on the torsion units in $\V(\Z G, N)$ by Algorithm~\ref{AlgorithmLocalGN} are strictly stronger than the restrictions provided by the HeLP Method.
\end{example}
\begin{proof}
Let $G=G_3(\alpha,\beta)$ for $\alpha$ and $\beta$ primitive elements of $\F_{25}$ and $\F_{49}$ respectively.
By Lemma~\ref{HeLPGpqd} it is enough to show that  $(x_0,x_1,x_2)=(2,0,-1)$ satisfies \eqref{HeLPInequality} for every $35$-th root of unity. To do that observe that if $r_0,r_1\in \F_p$ and $s_0,s_1\in \F_q$ then $\psi_0(r_0+r_1\alpha,s_0+s_1\beta)=\zeta_p^{r_1}\zeta_q^{s_1}$. Using this it is easy to prove that
\begin{eqnarray*}
\mu(\chi_0,m_i,\zeta_p^r\zeta_q^s) &= |\{(x,y): & 0\le x < p^2 - 1, 0\le y < q^2 - 1, \\
& & \alpha^x-r\alpha\in \F_p, \beta^y-s\beta\in \F_q, x\equiv y + i \bmod d \}|.
\end{eqnarray*}
After a computer aid calculation we get that for any $35$-th root of unity $\xi$ the triple $(\mu(\chi_0,m_0,\xi), \mu(\chi_0,m_1,\xi), \mu(\chi_0,m_2,\xi))$ is one of the following:
$( 4, 16, 8 )$, $( 8, 4, 16 )$, $( 8, 8, 8 )$, $( 10, 10, 10 )$, $( 10, 12, 13 )$, $( 12, 13, 10 )$, $( 13, 10, 12 )$, $( 16, 8, 4 )$.
All these tuples satisfy \eqref{HeLPInequality}.
\end{proof}

\section{Constructing metabelian counterexamples}\label{SecConstruction}

The aim of this section is analyzing how a negative solution to Sehgal's Problem, with $N$ and $G/N$ abelian, should be, with $N$ minimal in some sense.
By \cite[Corollary~5.3]{MargolisdelRioCW1}, $N$ should have at least two non-cyclic Sylow subgroups, so the minimal case is $N\cong C_{pq}\times C_{pq}$ with $p$ and $q$ different primes greater or equal than $7$, by Theorem~\ref{Positive2-3-5}.

The following observation will be useful for our constructions.

\begin{lemma}\label{Np'CqCq}
Assume $A \cong C_p \times C_p$ for some prime $p$ and that $G$ is an abelian group acting on $A$.
If $f(a)<0$ for some $f \in \mathcal{E}(G, H, A)$ with $H$ a subgroup of $G$ then $\Cen_G(a) = \Cen_G(A)$.
\end{lemma}

\begin{proof}
By Lemma~\ref{InequalitiesAbelianL}.\eqref{CyclicComplementL} $[G:\Cen_G(a)]>p$.
This implies that  $b = a^g \not\in \GEN{a}$ for some $g \in G$.
Therefore $A=\GEN{a,b}$.
Let $b\in \Cen_G(a)$. As $G$ is abelian we have $b^b = a^{gb} = a^{bg} = a^g = b$. This proves that $\Cen_G(a)\leq \Cen_G(b)$ and therefore
\[\Cen_G(A)=\Cen_G(\langle a,b \rangle)=\Cen_G(a)\cap \Cen_G(b)=\Cen_G(a),\] as desired.
\end{proof}

\subsection{General construction}\label{SubSecConstrucion}

We now explain how one can construct a group $G$ containing a normal nilpotent subgroup $N$ isomorphic to $C_{pq}\times C_{pq}$ such that $G/N$ is abelian for which our strategy fails to give a positive answer to Sehgal's Problem, in other words, so that the output of Algorithm~\ref{AlgorithmGlobalGN} with input $(G,N)$ is not \texttt{true}.
The idea is to combine the outputs $O_p$ and $O_q$ of Algorithm~\ref{AlgorithmLocalN} with input $N_p$ and $N_q$ respectively, in the following form:
First of all neither $O_p$ nor $O_q$ should be \texttt{true}.

Then we look for all the pairs $(K_p,f_p)$ and $(K_q,f_q)$ with $f_p\in E_p$ for some $(K_p,E_p) \in O_p$ with $K_p$ abelian and $f_q\in E_q$ for some $(K_q,E_q)\in O_q$ with $K_q$ abelian satisfying the following conditions:
\begin{enumerate}
\item There is $n\in N$ with $f_p(n_p)=f_q(n_q)\ne 0$, which will be fixed throughout.
\item\label{ApAq} There are abelian subgroups $T_p$ of $\Aut(N_p)$ and $T_q$ of $\Aut(N_q)$ with $K_p\le T_p$ and $K_q\le T_q$ such that:
    \begin{enumerate}
    \item\label{ApAqClass} if $f_p(x)\ne 0$ then $x\in n_p^{T_p}$ and if $f_q(x)\ne 0$ then $x\in n_q^{T_q}$;
    \item there is an isomorphism $\alpha:T_p/K_p \cong T_q/K_q$ such that for every $a\in T_p$ we have
        $$f_p(n_p^{aK_p})=f_q(n_q^{\alpha(a^{-1}K_p)}).$$
    \end{enumerate}
\end{enumerate}

Consider $T_p\times T_q$ as a subgroup of $\Aut(N)$ in the natural way and consider the natural maps $r_p:T_p\times T_q\rightarrow T_p$ (i.e. $r_p$ acts as restriction to $N_p$) and $\pi_p:T_p\rightarrow T_p/K_p$.
Define $r_q$ and $\pi_q$ similarly.
Then let
    $$\Gamma=\{\sigma \in T_p\times T_q : \alpha \pi_p r_p(\sigma)=\pi_q r_q(\sigma)\}$$
(i.e. $\Gamma$ is the pullback of $\alpha \pi_pr_p$ and $\pi_qr_q$) and consider the group $G=N\rtimes \Gamma$.
Observe that the maps $r_p:\Gamma\rightarrow T_p$ and $r_q:\Gamma\rightarrow T_q$ are surjective.
This implies that $x^G=x^{T_p}$ for every $x\in N_p$ and similarly $x^G=x^{T_q}$ for every $x\in N_q$.

We define $\varepsilon:N\rightarrow \Z$ as follows:
    \begin{equation}\label{Epsilonfpfq}
    \varepsilon(x) = \begin{cases}
    0, & \text{if } x_q\not\in n_q^G; \\ f_p(x_p^{g}), & \text{if } n_q=x_q^g, \text{ with }g\in G.
    \end{cases}
    \end{equation}
We now prove that $\varepsilon$ belongs to the output of Algorithm~\ref{AlgorithmGlobalGN} with input $G$ and $N$.
We have to show that $\varepsilon$ belongs to the intersection of outputs of Algorithm~\ref{AlgorithmLocalGN} with input $G$, $N$ and $p$ and with input $G$, $N$ and $q$.
By the definition of $\varepsilon$, to prove that $\varepsilon$ belongs to the first output it is enough to show that $f_p\in \mathcal{E}(G,\Cen_G(n_q),N_p)$.
That $f_q\in \mathcal{E}(G,\Cen_G(n_p),N_q)$ is proved in the same way after showing that the equality in \eqref{Epsilonfpfq} holds with the roles of $p$ and $q$ interchanged.

To prove that \eqref{Epsilonfpfq} holds also when interchanging $p$ and $q$ observe that by the definition of $\varepsilon$ and since $f_p$ is an $K_p$-class function we have $\varepsilon(x) = 0$, if $x_p \notin n_p^G$ or $x_q \notin n_q^G$. Otherwise, $x = n_p^{g_p} n_q^{g_q}$ for some $g_p,g_q\in \Gamma$ and we get
\begin{eqnarray*}
\varepsilon(x) &=& f_p\left(n_p^{g_p g_q^{-1}}\right) =
f_p\left(n_p^{\pi_p r_p(g_p g_q^{-1})}\right) =
f_q\left(n_q^{\alpha \pi_p r_p(g_qg_p^{-1})}\right) =
f_q\left(n_q^{g_q g_p^{-1}}\right) \\
&=& f_q\left(x_q^{g_p^{-1}}\right).
\end{eqnarray*}

Next to prove that $f_p \in \mathcal{E}(G, \Cen_G(n_q), N_p)$ observe that, by the pullback properties, we have
	$$\pi_p(\Cen_{\Gamma}(N_q))=\pi_p(\ker r_q) = \ker \pi_p = K_p.$$
Moreover, by Lemma~\ref{Np'CqCq}, we know $\Cen_G(N_q) = \Cen_G(n_q)$ and so $\Inn_{\Cen_G(n_q)}(N_p) = K_p$. Since furthermore all elements $x \in N_p$ satisfying $f_p(x) \neq 0$ lie in a $G$-conjugacy class  it is easy to check that $f_p \in \mathcal{E}(\Aut(N_p), K_p, N_p)$ implies $f_p \in \mathcal{E}(G, \Cen_G(n_q), N_p)$.

\begin{remark}
The construction above realizes all the possible partial augmentations of negative solutions to Sehgal's Problem for $G$ and $N$, in case $N\cong C_{pq}\times C_{pq}$ and $G/N$ abelian.
Indeed, if $u$ is a torsion unit of $\V(\Z G,N)$ which gives a negative solution then the map $\varepsilon$ defined on $N$ by $\varepsilon(n)=\pa{u}{n}$ belongs to the output of Algorithm~\ref{AlgorithmGlobalGN} with input $G$ and $N$.
This implies that there is $n\in N$, a $\Cen_G(n_q)$-class function $f_p$ on $N_p$ and a $\Cen_G(n_p)$-class function $f_q$ on $N_q$ such that \eqref{Epsilonfpfq} holds. Moreover, it also holds with the roles of $p$ and $q$ interchanged.
By Lemma~\ref{Np'CqCq} we have $\Cen_G(n_q)=\Cen_G(N_q)$ and $\Cen_G(n_p)=\Cen_G(N_p)$. Let $K_p=\Inn_{\Cen_G(N_q)}(N_p)$, $K_q=\Inn_{\Cen_G(N_p)}(N_q)$, $T_p=\Inn_{G}(N_p)$ and $T_q=\Inn_{G}(N_q)$.
Then $f_p\in \mathcal{E}(\Aut(N_p),K_p,N_p)$ and $f_q\in \mathcal{E}(\Aut(N_q),K_q,N_q)$.
Now it is easy to see that $K_p,f_p,T_p$ and $K_q,f_q,T_q$ satisfy all the conditions needed for the general construction above and clearly the map defined by \eqref{Epsilonfpfq} is the original map $\varepsilon$.
\end{remark}

\subsection{Concrete counterexamples}

To construct a concrete abelian group of automorphisms of $C_{pq}\times C_{pq}$ suitable for the construction from the previous subsection, we aim to better understand the possible abelian actions on elementary abelian $p$-groups of rank $2$ satisfying the conditions of Lemma~\ref{InequalitiesAbelianL}.
These are classified in the following Lemma.

\begin{lemma}\label{Dichotomy}
Let $A = C_p \times C_p$ for some prime $p$ and $T \leq \Aut(A)$ be abelian such that $T$ contains a subgroup $K$ of order bigger than $p$ and not divisible by $p$. Then exactly one of the following conditions holds:
\begin{enumerate}
	\item $T$ is cyclic of order dividing $p^2-1$ acting semiregularly on $A\setminus \{1\}$, i.e. no element in $T\setminus \{1\}$ fixes any $n \in A \setminus \{1\}$.
In this case $T$ is determined uniquely by its order up to conjugacy in $\Aut(A)$.
	\item $A$ contains two cyclic non-trivial subgroups invariant under the action of $T$.
Then $T$ is isomorphic to a non-cyclic subgroup of $C_{p-1}^2$.
\end{enumerate}
\end{lemma}

\begin{proof}
We identify $A$ with the underlying additive group of $\F_p^2$.
In this way we consider $T$ as a subgroup of $\GL(2,p)$.
This can be used to prove that conditions (1) and (2) are not compatible, for if condition (2) holds after an appropriate choice of basis, one may assume that $T$ is contained in the group formed by the diagonal matrices in $\GL(2,p)$. Clearly these matrices do not contain a cyclic group of order bigger than $p$.

Suppose that condition (2) does not hold. This means precisely that there is no element $a \in \GL(2,p)$ such that $a^{-1}Ta$ is formed by diagonal matrices.
Suppose that the action of $T$ is not semiregular.
Then $1$ is an eigenvalue of some non-trivial element $g$ of $T$.
If $g$ is diagonalizable then there is $a\in \GL(2,p)$ such that $a^{-1}ga=\diag(1,x)$ for some $x\in \F_p\setminus \{0,1\}$. As $T$ is abelian $a^{-1}Ta$ is formed by diagonal matrices, contradicting the assumption.
Thus there is $a\in \GL(2,p)$ such that $a^{-1}ga=\pmatriz{{cc} 1&1\\0&1}$.
Again, using that $T$ is abelian we deduce that $a^{-1}Ta$ is contained in the maximal abelian subgroup $M$ formed by the upper triangular matrices of $\GL(2,p)$.
Thus $g^{-1}Kg$ is a subgroup of $M$ of order coprime with $p$. As $|M| = (p-1)p$ we deduce that $|K|\le p-1$, a contradiction.
Thus $T$ acts semiregularly on $A$.
By \cite[II, 7.3 Satz ]{HuppertI} and its proof, identifying $A$ with the underlying additive group of $\F_{p^2}$ we deduce that $T$ is a subgroup of the group of automorphism of $\F_{p^{2}}$ given by multiplication by non-zero elements of $\F_{p^2}$ and hence $T$ is a cyclic group of order dividing $p^2-1$. This cyclic group is unique up to conjugacy in $\Aut(A)$. \end{proof}

Let again $A = C_p \times C_p$ for some prime $p$. By Lemma~\ref{InequalitiesAbelianL} we know that when studying the possibly non-empty $\mathcal{E}(\Aut(A), K, A)$ for $K$ abelian we can restrict ourself to the study of $K$ being of order bigger than $p$ and non-divisible by $p$. Then, by Lemma~\ref{Dichotomy}, $K$ is contained, up to conjugacy, in one of two maximal abelian subgroups $\mathcal{C}$ or $\mathcal{D}$ of $\Aut(A)$.
Here $\mathcal{C}$ is cyclic of order $p^2-1$ and acting regularly on the non-trivial elements of $A$, while $\mathcal{D}$ stabilizes two cyclic subgroups of $A$ of order $p$ and is isomorphic to $C_{p-1} \times C_{p-1}$.
In both cases $K$ does not contain the centre of $\Aut(A)$ by Lemma~\ref{InequalitiesAbelianL}.\eqref{NotMultipleOfqL}.
Calculating all such $K$ is quite straightforward.
Unfortunately we have not been able to systematically produce $\mathcal{E}(\Aut(A), K, A)$ for all such $K$.
We now provide all non-trivial $\mathcal{E}(\Aut(A), K, A)$ for $p \leq 19$.
The calculations have been performed using the computer algebra system \texttt{GAP} \cite{GAP}. The inequalities were solved using \texttt{4ti2} \cite{4ti2}, or rather the \texttt{GAP}-package \cite{4ti2Interface} providing an interface to use it.

We explain how to read the tables.
In Table~\ref{TableCyclic} we list the subgroups $K$ of $\mathcal{C}$ with non-empty $\mathcal{E}(\Aut(A), K, A)$.
As each $K$ is determined by its order we simply display its isomorphism type.
We represent the elements $f\in \mathcal{E}(\Aut(A), K, A)$ as follows:
We fix a generator $c$ of $\mathcal{C}$ and set $r=[\mathcal{C}:K]$.
Then the $K$-classes of $N\setminus \{1\}$ are $X,X^c,X^{c^2},\dots,X^{c^{r-1}}$ for any given $K$-class $X$.
We describe $f$ by displaying the list $(f(X),f(X^c),f(X^{c^2}),\dots,f(X^{c^{r-1}}))$.
This determines $f$ up to the election of $X$, or equivalent up to a cyclic permutation of the list.
We hence list the values of $f$ only up to cyclic permutations.
Of course choosing another generator of $\mathcal{C}$ will change the ordering of the values, but all these orderings can be read of from the same list of values.

In Table~\ref{TableNonCyclic} we display the non-trivial outputs in case that $K$ is not cyclic and hence we may assume that it is a subgroup of $\mathcal{D}$.
It turns out that in all cases $\mathcal{D}/K$ is cyclic.
So, as in the previous case, we list tuples $(f(X),f(X^c),f(X^{c^2}),\dots,f(X^{c^{r-1}}))$ with $c$ a generator of $\mathcal{D}/K$ and we only list tuples up to cyclic permutation.
Here however the isomorphism type of $K$ is not sufficient to describe it uniquely and so we provide explicit generators. Here it is understood that after writing $A=\GEN{n}\times \GEN{m}$ such that $\GEN{n}$ and $\GEN{m}$ are stabilized by $\mathcal{D}$, the type $(a,b),(c,d)$ of $K$ implies that it is generated by the automorphisms $(n,m) \mapsto (n^a,m^b)$ and $(n,m) \mapsto (n^c, m^d)$.

\begin{table}[h]
\begin{tabular}{l|l|l|l}
$p$ & $K$ & $\mathcal{C}/K$ & $f$  \\ \hline \hline
7 & $C_{16}$ & $C_3$ & \; $(-1,2,0)$ \\ \hline
11 & $C_{24}$ & $C_5$ & \; $(-2,2,1,0,0), (-1,1,1,0,0) $ \\ \hline
13 & $C_{28}$ & $C_6$ & \; $(-1,1,1,0,0,0), (-1,0,1,1,0,0)$ \\
 & $C_{42}$ & $C_4$ & \; $(-1,1,1,0)$ \\
 & $C_{56}$ & $C_3$ & \; $(-2,2,1),(-1,1,1)$ \\ \hline
17 & $C_{36}$ & $C_8$ & \; $(-1,1,1,0,0,0,0,0)$ \\
 & $C_{72}$ & $C_4$ & \; $(-1,1,1,0),(-1,0,1,1)$ \\ \hline
19 & $C_{40}$ & $C_9$ & $\matriz{{l} (-2,-2,2,0,2,0,0,0,1), (-1,-1,1,0,1,0,0,0,1), \\
                        (-1,1,0,1,0,0,0,0,0)}$ \\
 & $C_{60}$ & $C_6$ & \; $(-1,1,1,0,0,0),(-1, 1, 0, 0, 1, 0) $ \\
 & $C_{120}$ & $C_3$ & \; $(-2,3,0), (-1,1,1), (-1,2,0)$
\end{tabular}
\caption{\label{TableCyclic}
	List of all $f \in \mathcal{E}(\Aut(C_p\times C_p), K, C_p \times C_p)$ with $K$ cyclic, up to cyclic permutation.}
\end{table}

\begin{table}[h]
\begin{tabular}{l|l|l|l}
$p$ & $K$ & $\mathcal{D}/K$ & \; $f$  \\ \hline \hline
13 & $(6,3),(1,5), C_{12} \times C_4$ & $C_3$ & \; $(-1,2,0)$ \\ \hline
17 & $(3,9),(1,16), C_{16} \times C_2$ & $C_8$ & \; $ (-1, 1, 1, 0, 0, 0, 0 ,0)$ \\
 & $(3,11),(1,16), C_{16} \times C_2$ & $C_8$ &\; $(-1, 1,1,0,0,0,0,0)$  \\
 & $(3,10), (1,13), C_{16} \times C_4$ & $C_4$ & \; $(-1, 1, 1, 0)$ \\
 & $(3,9), (1,13), C_{16} \times C_4$ & $C_4$ & \; $(-1, 1, 1, 0)$ \\ \hline
19 & $(2,9),(1,18), C_{18} \times C_2$  & $C_9$ & \; $(-1, 1,0, 1, 0, 0, 0,0,0)$ \\
 & $(2,4),(1,18), C_{18} \times C_2$  & $C_9$ & \; $(-1, 1,0, 1, 0, 0, 0,0,0)$ \\
 & $(2,6),(1,18), C_{18} \times C_2$  & $C_9$ & \; $(-1, 1,0, 1, 0, 0, 0,0,0)$ \\
 & $(2,15),(1,7), C_{18} \times C_3$  & $C_6$ & \; $(-1, 0, 1, 0, 1, 0)$ \\
 & $(2,4),(1,7), C_{18} \times C_3$  & $C_6$ & \; $(-1, 0, 1, 0, 1, 0)$ \\
 & $(2,6),(7,18), C_{18} \times C_6$  & $C_3$ & \; $(-1, 2, 0), (-1, 1, 1)$ \\ \hline
\end{tabular}
\caption{\label{TableNonCyclic}
List of all $f \in \mathcal{E}(\Aut(C_p \times C_p), K, C_p \times C_p)$ with $K$ stabilizing two cyclic subgroups in $C_p \times C_p$, up to cyclic permutation.}
\end{table}

Using these tables and our construction one can easily construct groups where our strategy fails.
The smallest such group would be by choosing $N = C_{7 \cdot 13} \times C_{7 \cdot 13}$, $K_7 = C_{16}$ and $K_{13} = C_{12} \times C_4$ (the first ones in Tables~\ref{TableCyclic} and \ref{TableNonCyclic}).
In this case $T_7/K_7 \cong T_{13}/K_{13} \cong C_3$.
By the values of $f_p$ and $f_q$ given in the tables we conclude that in this case our construction can be carried through.

If one assumes that both $K_p$ and $K_q$ in the construction are cyclic then we obtain groups of the type defined in Notation~\ref{Gpqd}.
More precisely we get that there are groups of type $G_3(7,19)$, $G_6(13,19)$, $G_4(13,17)$ and $G_3(13,19)$ which are all counterexamples to our strategy.
Actually one of those of first type is the counterexample to the Zassenhaus Conjecture in \cite{EiseleMargolis}.
But also positive results can be read of from the tables, it follows i.e. that if $N \cong C_{11\cdot p} \times C_{11 \cdot p}$ with $p \leq 19$ and $G/N$ is abelian then Sehgal's Problem has a positive solution for $G$ and $N$.

\begin{remark}\label{RemFinal}
We explain why the construction in this section provided good candidates for counterexamples to the Zassenhaus Conjecture, which lead Eisele and Margolis to prove that actually such counterexamples exist.

Let $C$ be a cyclic group of order $n$ and assume we want to find a unit $u$ in $\V(\Z G)$ of order $n$ which is a counterexample to the conjecture. By the well known double action formalism, as explained in \cite[Section 2]{EiseleMargolis} or \cite[38.6]{Sehgal1993}, this is equivalent to showing that there exists a $\Z(G \times C)$-module $M$ of rank $|G|$ such that $M$ is free as $\Z G$-module and such that the character associated to $M$ takes negative values on certain elements.

Now assume we have a non-trivial distribution of partial augmentations which is in the output of Algorithm~\ref{AlgorithmLocalGN} with input $G$ and $N$. Then these partial augmentations also satisfy the constraints of the HeLP method by Remark~\ref{HeLPRemark} and this implies the existence of a $\Q(G \times C)$-module of rank $|G|$ which is free as $\Q G$-module and whose associated character has negative values. On the other hand this also implies there is a $\Z(N \times C)$-lattice of rank $|G|$ which is locally free as $\Z N$-lattice and whose associated character has negative values by \cite[Theorem 3.4]{MargolisdelRioCW1}, \cite[Theorem 3.3]{CliffWeiss}.

This observation and the concrete construction described above gave rise to the idea that such distributions of partial augmentations might correspond to  semilocal counterexamples to the Zassenhaus Conjecture, by "matching" the two modules described above. As shown in \cite[Section 5]{EiseleMargolis} this is indeed the case and under a few more conditions this even allows to find global counterexamples \cite[Section 6]{EiseleMargolis}. Hence the algorithms described here might e.g. be used to further investigate the boundary between groups satisfying the Zassenhaus Conjecture and not satisfying it.
\end{remark}

\medskip

\textbf{Acknowledgement:} We are thankful to F. Eisele for fruitful discussions.

\bibliographystyle{amsalpha}
\bibliography{CW}

\noindent
Departamento de Matemáticas, Universidad de Murcia, 30100 Murcia, Spain\newline
email: leo.margolis@um.es, adelrio@um.es

\end{document}